\documentclass[12pt]{article}

\usepackage[margin=1in]{geometry}
\usepackage{amsmath,amsthm,amssymb,bbm,color,enumitem,esint,esvect,float,graphicx,mathrsfs, esint, mathtools, url}

\usepackage{xcolor}

\usepackage{tikz}
\usetikzlibrary{calc}
\usepackage{hyperref,mathtools}

\usepackage{amsxtra}
\usepackage{comment}
\usepackage{amsthm,mathrsfs}
\usepackage{amsmath,amsthm,amssymb,amsrefs,bbm,color,enumitem,esint,esvect,float,graphicx,mathrsfs}

\newtheorem{theorem}{Theorem}[section]
\newtheorem{lemma}[theorem]{Lemma}

\theoremstyle{definition}
\newtheorem{definition}[theorem]{Definition}

\newtheorem{proposition}[theorem]{Proposition}
\newtheorem{corollary}[theorem]{Corollary}

\theoremstyle{remark}
\newtheorem{remark}[theorem]{Remark}

\numberwithin{equation}{section}

\usepackage{epstopdf}
\AppendGraphicsExtensions{.tif}

 \usepackage{authblk}

\title{Notes on Bilinear Bessel Potentials}

\author{Ana Čolović \qquad Xinyu Gao\\University of Missouri}

\newcommand{\Addresses}{{
  \bigskip
  \footnotesize

  A.~Čolović, \textsc{Department of Mathematics, University of Missouri,
    Columbia, Missouri 65201}\par\nopagebreak
  \textit{E-mail address}, A.~Čolović: \texttt{acg7y@umsystem.edu}

  \medskip

  X.~Gao ,  \textsc{Department of Mathematics, University of Missouri,
    Columbia, Missouri 65201}\par\nopagebreak
  \textit{E-mail address}, X.~Gao: \texttt{xggh8@umsystem.edu}
}}

\begin{document}

\maketitle

\begin{abstract}
In analogy with bilinear Riesz potentials, we introduce bilinear Bessel potentials and characterize their boundedness from $L^p\times L^q$
 into Lebesgue and Lorentz spaces $L^{r,\alpha}.$ In several cases we identify the optimal Lorentz indices by constructing explicit counterexamples.
\end{abstract}

\section{Introduction}

Fractional integral operators are fundamental objects in harmonic analysis,
closely related to behavior of functions in Sobolev spaces. In the linear setting,
the mapping properties of Riesz and Bessel fractional integral operators, otherwise known as Riesz and Bessel potentials have long been established and appear, for example, in \cite{Ste}. Bilinear Riesz potential was introduced by Grafakos in \cite{GraSt}. Bilinear fractional operators exhibit new structural phenomena, including subtle endpoint behavior, first observed in the works of Grafakos–Kalton (in \cite{GraKal}) and Kenig and Stein (in \cite{KenSte}).

Motivated by the theory of bilinear Riesz potentials we introduce and study \emph{bilinear Bessel potentials}. Let $n\in \mathbb{N},$ where $n\geq 1,$ and let $0<s<n$.
We let $G_s$ denote the Bessel kernel whose Fourier transform satisfies
\[
\widehat{G_s}(\xi)=(1+4\pi^2|\xi|^2)^{-s/2}.
\]
The associated linear Bessel potential of order $s$ is given by
\[
J_s f(x)
=
(G_s * f)(x)
=
\int_{\mathbb R^n} G_s(x-y)\,f(y)\,dy,
\]
while the bilinear Bessel potential is defined by
\[
\mathcal J_s(f,g)(x)
=
\int_{\mathbb R^n} G_s(y)\,f(x-y)\,g(x+y)\,dy.
\]
These operators may be viewed as bilinear analogues of the classical Bessel
potential, interpolating between bilinear convolution operators with
integrable kernels and bilinear fractional integrals. 

The main goal of this paper is to characterize the boundedness of
$\mathcal J_s$ from
$L^p(\mathbb R^n)\times L^q(\mathbb R^n)$ into Lebesgue and Lorentz spaces.
Using scaling considerations and explicit counterexamples, we prove that
boundedness can only occur when
\[
\frac1p+\frac1q-\frac{s}{n}\le \frac1r\le \frac1p+\frac1q,
\]
and that this range is sharp up to certain endpoint phenomena. In particular,
the critical line $\frac1p+\frac1q=\frac{s}{n}$ fails for mapping into
$L^\infty$, as shown in Section~\ref{prop:fail-critical-Linfty}. Within this region, we establish strong and
weak-type bounds, including Lorentz space refinements obtained via restricted
weak-type interpolation. In several cases, we prove that the Lorentz indices
are optimal by constructing concrete counterexamples. We refer to exponents satisfying
$\frac1r=\frac1p+\frac1q$
as the \emph{Lebesgue indices}, and those satisfying
$\frac1r=\frac1p+\frac1q-\frac{s}{n}$
as the \emph{fractional indices}.

\section{Preliminaries and Background}
In this section, we recall necessary preliminaries.

\subsection{Lorentz spaces and tools}
We work with Lorentz spaces, are defined as follows (see \cite{Gra}, Section~1.4.2).
\begin{definition}[Lorentz spaces]
Let $(X,\mu)$ be a measure space. For a measurable function $f$ on $X$, define
\[
d_f(\alpha)=\mu(\{|f|>\alpha\}), \qquad
f^*(t)=\inf\{\alpha>0:\ d_f(\alpha)\le t\}.
\]
For $0<p<\infty$ and $0<q\le\infty$, the Lorentz space $L^{p,q}(X,\mu)$ consists of
all measurable functions $f$ such that
\[
\|f\|_{L^{p,q}}<\infty,
\]
where
\[
\|f\|_{L^{p,q}}
=
\begin{cases}
\left(\displaystyle\int_0^\infty
(t^{1/p}f^*(t))^q\,\dfrac{dt}{t}\right)^{1/q}, & 0<q<\infty,\\[1em]
\displaystyle\sup_{t>0} t^{1/p}f^*(t), & q=\infty.
\end{cases}
\]
\end{definition}

We also use the following proposition, which appears as Proposition~1.4.9 on p.~53 of \cite{Gra}. 
\begin{proposition} 
    For $0<p<\infty,$ $0<q\leq \infty,$ the following formula holds:
  \begin{equation}\label{eq:distr-def-lorentz}
    \|f\|_{L^{p,q}}=p^{\frac{1}{q}}\left(\int_{0}^{\infty}[d_f(\lambda)^{\frac{1}{p}}\lambda]^q\frac{d\lambda}{\lambda}\right)^{\frac{1}{q}},
  \end{equation}
with the usual modification when $q=\infty$.
The following inequality represents an extension of Young's inequality in Lorentz spaces. It was proved by O'Neil and can be found in \cite{ONe}.
\end{proposition}

\begin{lemma}(O'Neil's inequality)
    Let $f\in L^{p,\alpha_1}(\mathbb{R}^n),$ $g\in L^{q,\alpha_2}(\mathbb{R}^n),$ with $1<p,q,r<\infty,$ and $0<\alpha_1,\alpha_2\leq \infty$ such that \[\frac{1}{r}+1=\frac{1}{p}+\frac{1}{q} \qquad \text{and} \quad \frac{1}{\alpha}\leq \frac{1}{\alpha_1}+\frac{1}{\alpha_2},\]
   and $\alpha>0.$
   Then $f*g\in L^{r,\alpha}(\mathbb{R}^n),$ and 
   \[
   \|f*g\|_{L^{r,\alpha}} \leq 3r \|f\|_{L^{p,\alpha_1}}\|g\|_{L^{q,\alpha_2}}.
   \]
\end{lemma}
\subsection{Properties of Bessel Kernels}
Combining Proposition~6.1.5 on p.~6 of \cite{GraMod} with the estimate on p.~133 of \cite{Ste}, we obtain the pointwise estimates on the Bessel kernel, stated in the following lemma. 
 
\begin{lemma}[Integrability of the Bessel kernel] 
\label{lem:bessel-kernel-Lr}
Let $0<s<n$, and let $G_s$ denote the Bessel kernel on $\mathbb R^n$ of order $s$.
Then there exist constants $C,c>0$ (depending only on $n$ and $s$) such that
\begin{equation}\label{eq:Gs-pointwise}
G_s(x)\le C\,|x|^{s-n}\quad \text{for }0<|x|\le 1,
\qquad
G_s(x)\le C\,e^{-c|x|}\quad \text{for }|x|\ge 1.
\end{equation}
Consequently,
\begin{equation}\label{eq:Gs-Lr-range}
G_s\in L^r(\mathbb R^n)
\quad\text{for every}\quad
1\le r<\frac{n}{n-s},
\end{equation}
and
\[
G_s\in L^{\frac{n}{n-s},\infty}(\mathbb R^n).
\]
\end{lemma}

\section{Necessary Conditions and Exponent Geometry}

We determine the possible range of exponents for boundedness of the bilinear
Bessel potentials using dimensional analysis.

\begin{proposition}
\label{prop:necessary-conds-weak}
Let $0<s<n$.
Fix $1\le p,q\le \infty$ and $0<r<\infty$.  

If there exists a constant $C>0$ such that
\[
\|\mathcal J_s(f,g)\|_{L^{r,\infty}}
\le C\,\|f\|_{L^p}\,\|g\|_{L^q}
\qquad\text{for all }f,g\in \mathcal S(\mathbb R^n),
\]
then necessarily
\begin{equation}\label{eq:necessary-strip}
\frac1p+\frac1q-\frac{s}{n}\ \le\ \frac1r\ \le\ \frac1p+\frac1q.
\end{equation}
\end{proposition}

\begin{proof}
Since $G_s\ge 0$ and is not identically zero, there exists a $c_0>0$ such that
\[
c_0:=\int_A G_s(y)\,dy>0,\qquad\text{where } A:=\{y\in\mathbb R^n: 1/2\le |y|\le 1\}.
\]
For $R\ge 3$, set $f_R=\mathbf 1_{B(0,R)}$ and $g_R=\mathbf 1_{B(0,R)}$.
If $x\in B(0,R-2)$ and $y\in A$, then
\[
|x\pm y|\le |x|+|y|\le (R-2)+1<R.
\]
Hence $f_R(x-y)=g_R(x+y)=1$ for all $y\in A$, and $x\in B(0, R-2).$ Therefore,
\[
\mathcal J_s(f_R,g_R)(x)\ \ge\ \int_A G_s(y)\,dy\ =\ c_0
\qquad\text{for all }x\in B(0,R-2).
\]
By the definition of the weak-$L^r$ quasi-norm,
\[
\|\mathcal J_s(f_R,g_R)\|_{L^{r,\infty}}
\ge c_0\,|B(0,R-2)|^{1/r}\ \gtrsim\ R^{n/r}.
\]
On the other hand, $\|f_R\|_{L^p}\simeq R^{n/p}$ and $\|g_R\|_{L^q}\simeq R^{n/q}$.
If the stated weak-type estimate held, we would obtain for all large $R$
\[
R^{n/r}\ \lesssim\ R^{n/p+n/q},
\]
which forces $\frac1r\le \frac1p+\frac1q$.

To obtain the lower bound
\(
\frac1r \ge \frac1p+\frac1q-\frac{s}{n},
\)
we choose nonnegative $f,g\in C_c^\infty(\mathbb R^n)$ supported in $B(0,1)$ and satisfying
$f\equiv g\equiv 1$ on $B(0,1/2)$. For $\lambda\ge 1$ define the $L^p$- and $L^q$-normalized dilates
\[
f_\lambda(x):=\lambda^{n/p}f(\lambda x),
\qquad
g_\lambda(x):=\lambda^{n/q}g(\lambda x),
\]
so that $\|f_\lambda\|_{L^p}=\|f\|_{L^p}$ and $\|g_\lambda\|_{L^q}=\|g\|_{L^q}$.

Let $x\in \mathbb R^n$ satisfy $|x|\le \frac{1}{4\lambda}$ and let
$y\in \mathbb R^n$ with $|y|\le \frac{1}{4\lambda}$.
Then
$|x\pm y|\le |x|+|y|\le \frac{1}{2\lambda}$.

By the choice of $f$ and $g$, this implies
\[
f_\lambda(x-y)=\lambda^{n/p},
\qquad
g_\lambda(x+y)=\lambda^{n/q}.
\]

Using the lower bound $G_s(y)\gtrsim |y|^{s-n}$ for $|y|\le 1$, we obtain
\begin{align*}
\mathcal J_s(f_\lambda,g_\lambda)(x)
&=\int_{\mathbb R^n} G_s(y)\,f_\lambda(x-y)\,g_\lambda(x+y)\,dy\\
&\ge
\int_{|y|\le 1/(4\lambda)} G_s(y)\,f_\lambda(x-y)\,g_\lambda(x+y)\,dy\\
&\gtrsim
\lambda^{n/p+n/q}
\int_{|y|\le 1/(4\lambda)} |y|^{s-n}\,dy\simeq
\lambda^{n/p+n/q}\cdot \lambda^{-s}
=
\lambda^{\,n/p+n/q-s}.
\end{align*}

Hence, the superlevel set
\[
E_\lambda:=\Bigl\{x:\ \mathcal J_s(f_\lambda,g_\lambda)(x)\gtrsim \lambda^{n/p+n/q-s}\Bigr\}
\]
contains a ball of radius $\simeq \lambda^{-1}$, and so $|E_\lambda|\gtrsim \lambda^{-n}$.
By the definition of the weak-$L^r$ quasi-norm,
\begin{align*}
\|\mathcal J_s(f_\lambda,g_\lambda)\|_{L^{r,\infty}}
\ &\gtrsim\ \lambda^{n/p+n/q-s}\,|E_\lambda|^{1/r}
\ \gtrsim\ \lambda^{n/p+n/q-s}\cdot (\lambda^{-n})^{1/r}\\
\ &=\ \lambda^{\,n/p+n/q-s-n/r}.
\end{align*}
Therefore,
letting $\lambda\to\infty$, the bound
\[
\lambda^{\,n/p+n/q-s-n/r}\lesssim 1
\]
forces
\[
\frac{1}{r}\ge \frac{1}{p}+\frac{1}{q}-\frac{s}{n}.
\]
This completes the proof.
\end{proof}

\section{Strong-Type Lebesgue Bounds}
\subsection{\textbf{$L^p\times L^q\to L^r$}}

To treat the Lebesgue range
\[
\frac1p+\frac1q=\frac1r,\qquad p,q\ge1,
\]
we recall a version of a theorem of Grafakos and Soria (in \cite{GraSor}).  
Given a nonnegative regular Borel measure $\mu$ on $\mathbb{R}^n\times\mathbb{R}^n$, set
\[
T_\mu(f,g)(x)
=
\int_{\mathbb{R}^n\times\mathbb{R}^n}
f(x-y)g(x-z)\,d\mu(y,z).
\]

\begin{theorem}[Grafakos--Soria {\cite{GraSor}}]\label{thm:grafakos-soria}
Let $1/p+1/q=1/r\leq1$. For nonnegative $\mu$, if $\mu$ is a finite measure, the following hold:
 $T_\mu:L^p(\mathbb{R}^n)\times L^q(\mathbb{R}^n)\to L^r(\mathbb{R}^n)$ and
 $T_\mu:L^p(\mathbb{R}^n)\times L^q(\mathbb{R}^n)\to L^{r,\infty}(\mathbb{R}^n)$.

\end{theorem}

\begin{corollary}\label{cor:bessel-strong}
Let $s<n$ and $1/p+1/q=1/r\le1$. Then the bilinear Bessel potential satisfies
$\mathcal J_s:
L^p(\mathbb{R}^n)\times L^q(\mathbb{R}^n)
\longrightarrow
L^r(\mathbb{R}^n)$.
\end{corollary}

\begin{proof}
Let $G_s$ be the Bessel kernel and define $\mu(y,z)=\delta_0(y+z)G_s(y)$, where $\delta_0$ is Dirac delta function. Then
\[
\mathcal J_s(f,g)(x)
=
\int f(x-y)g(x-z)\,d\mu(y,z).
\]
Since $G_s\in L^1(\mathbb{R}^n)$ and $G_s\ge0$, $\mu$ is finite, and the claim follows from Theorem~\ref{thm:grafakos-soria}.\end{proof}

\subsection{$L^1\times L^1\to L^{\frac{1}{2}}$}
The following lemma proves the $L^1\times L^1\to L^{\frac{1}{2}}$ boundedness of the Bessel operators, allowing us to extend Corollary \ref{cor:bessel-strong}.
\begin{lemma}
The operator $\mathcal{J}_s$ is bounded between $L^1(\mathbb{R}^n)\times L^1(\mathbb{R}^n)\to L^\frac{1}{2}(\mathbb{R}^n).$
\end{lemma}
\begin{proof}

For $f, g \in L^1(\mathbb{R}^n), $ and $k\in \mathbb{Z}$, we define
\[
B_k(f,g)\coloneq \int_{|y|<2^{-k}}f(x-y)g(x+y)\, dy.
\]

From Lemma~5(iii) of \cite{KenSte}, we know that

\[
\|B_k(f,g)\|_{L^{\frac{1}{2}}}\lesssim 2^{-nk}\|f\|_{L^1}\|g\|_{L^1}.
\]

We decompose  $\mathcal{J}_s(f,g)(x),$ for $x\in \mathbb{R}^n$ in the following way:

\begin{align*}
    |\mathcal{J}_s(f,g)(x)| &\leq \sum_{k\in \mathbb{Z}} \int_{2^{-k-1}<|y|\leq 2^{-k}}|f(x-y)||g(x+y)|G_s(y) \, dy. \\
    &\lesssim\sum_{k< 0} e^{-\frac{2^{-k}}{4}}B_k(|f|,|g|)(x)+ \sum_{k\geq 0} 2^{(-k(s-n))}B_k(|f|,|g|)(x) \\
    &=\sum_{k\in \mathbb{Z}} a_k \, B_k(|f|,|g|),
\end{align*}

where, \[ a_k = \begin{cases} 
          2^{k(n-s)} & k\geq 0 \\
          e^{-\frac{2^{-k}}{4}} & k< 0. 
       \end{cases}
    \]

Since we can pass the power $\frac{1}{2}$ inside the sum, we have

\begin{align*}
\int_{\mathbb{R}^n}|\mathcal{J}_s(f,g)(x)|^{\frac12}dx &\lesssim \int_{\mathbb{R}^n} \left|\sum_{k\in \mathbb{Z}}a_k  \, B_k(|f|,|g|)(x)\right|^\frac{1}{2} \, dx \\&\lesssim \int_{\mathbb{R}^n}\sum_{k\in \mathbb{Z}}a_k^{\frac{1}{2}} B_k(|f|,|g|)(x)^{\frac{1}{2}}\, dx \\
&\lesssim \sum_{k\in \mathbb{Z}}a_k^{\frac{1}{2}} \|B_k(|f|,|g|)\|_{L^{\frac{1}{2}}}^{\frac12} \lesssim \sum_{k\in \mathbb{Z}}a_k^{\frac{1}{2}} 2^{\frac{-nk}{2}}\|f\|^{\frac12}_{L^1}\|g\|^{\frac12}_{L^1} \\
&=\|f\|_{L^1}^{\frac12}\|g\|_{L^1}^{\frac12} \sum_{k\in \mathbb{Z}}a_k^{\frac{1}{2}}2^{\frac {-nk}{2}} \\
&=\|f\|_{L^1}^{\frac{1}{2}}\|g\|_{L^1}^{\frac12} \left(\sum_{k\geq 0} 2^{k(\frac{n-s}{2}-\frac{n}{2})}  + \sum_{k<0} e^{-\frac{2^{-k}}{4}}2^{\frac{-nk}{2}}\right).
\end{align*}

The first summand converges since it is a geometric series and the second summand converges since the limit of the subsequent terms is equal to 0. \end{proof}

\begin{remark}
Combining the two results established in this section and applying the interpolation theorem (see \cite{GraMod}, Proposition~7.2.1), we obtain that for all $1\le p,q\le\infty$ satisfying
\[
\frac1p+\frac1q=\frac1r,
\]
the operator $\mathcal J_s$ extends boundedly as
\[
\mathcal J_s:\ L^p(\mathbb R^n)\times L^q(\mathbb R^n)\longrightarrow L^r(\mathbb R^n).
\]
\end{remark}

\section{Lorentz Space Estimates on the Fractional Index Surface}

\subsection{ $L^1\times L^q \longrightarrow L^{r_1,\alpha}$ for $1<q<\infty$}\label{sec:L1LQ}

Kenig and Stein proved the following result on the boundedness of bilinear Riesz potentials $\mathcal{I}_s$ (see \cite{KenSte}, Theorem~2, p.~4).

\begin{lemma}[$L^p\times L^1$ and $L^1\times L^q$ Lorentz boundedness, Kenig--Stein (1999)]
\label{lem:L1LQ}
Let $0<s<n$.
\begin{enumerate}
\item \label{item:lp-l1}  Suppose
\[
1<p<\infty,
\qquad
\frac{1}{p}+1=\frac{1}{r_1}+\frac{s}{n}.
\]
Then, for every $\alpha_1\ge p$,
\[
\mathcal{I}_s:
L^p(\mathbb R^n)\times L^1(\mathbb R^n)
\longrightarrow
L^{r_1,\alpha_1}(\mathbb R^n)
\]
is bounded. In particular, there exists $C>0$ such that
\[
\|\mathcal{I}_s(f,g)\|_{L^{r_1,\alpha_1}}
\le C\,\|f\|_{L^p}\,\|g\|_{L^1}.
\]
Moreover, the same bounds hold after interchanging the two input spaces.
\item \label{item:l1-l1}
Let $0<s<n$, and set
\[
r_{**}=\frac{n}{2n-s}.
\]
Then the bilinear Riesz potential operator $\mathcal I_s$ satisfies
\[
\mathcal I_s:
L^1(\mathbb R^n)\times L^1(\mathbb R^n)
\longrightarrow
L^{r_{**},\infty}(\mathbb R^n).
\]
More precisely, there exists a constant $C>0$ such that
\[
\|\mathcal I_s(f,g)\|_{L^{r_{**},\infty}}
\le C\,\|f\|_{L^1}\,\|g\|_{L^1},
\qquad
f,g\in L^1(\mathbb R^n).
\] 
\end{enumerate}
\end{lemma}

\begin{remark}
Recall that the Bessel kernel satisfies
\[
0\le G_s(x)\le C\,|x|^{s-n}=C\,I_s(x),
\qquad x\in\mathbb R^n.
\]
Consequently, for all $f,g\ge0$,
\begin{equation} \label{eq:riesz-dominates}
|\mathcal{J}_s(f,g)(x)|
\le C\,\mathcal{I}_s(|f|,|g|)(x).
\end{equation}
By replacing $f,g$ with $|f|,|g|$, the same inequality holds for general $f,g$.
Since the bounds \eqref{item:lp-l1} and \eqref{item:l1-l1} have been shown to hold for the bilinear Riesz potential operator,
the corresponding bounds hold for the operator $\mathcal{J}_s$.
\end{remark}

\subsection{Interior of the Fractional Index Plane}

In this subsection we prove interpolation results that characterize the boundedness of the bilinear operator for indices on the Fractional Index plane, as defined in the following theorem. 

We first introduce the following useful theorem (see Theorem 7.2.2 in \cite{GraMod}).
\begin{theorem}[Lorentz boundedness on the interpolation surface via three-point restricted weak type]
\label{lem:convex-hull-surface}
Let $0<s<n$ and let $T$ be a bilinear operator which is quasi-sublinear in each variable. 
Assume that $T$ is of restricted weak type $(p_k,q_k,r_k)$ for $k=0,1,2$, i.e.
there exist constants $C_k>0$ such that for all measurable $E,F\subset\mathbb R^n$
of finite measure and all $\lambda>0$,
\begin{equation}\label{eq:rwk-3pts}
\lambda\,
\bigl|\{x\in\mathbb R^n:\ |T(\mathbf 1_E,\mathbf 1_F)(x)|>\lambda\}\bigr|^{1/r_k}
\le C_k\,|E|^{1/p_k}\,|F|^{1/q_k},
\end{equation} and assume that the points $(p_k,q_k,r_k),$ for $k=1,2,3$ are not collinear.

Let $(p,q,r)$ satisfy $1\le p,q\le\infty$ and $0<r<\infty$.
Assume moreover that the point $\Bigl(\frac1p,\frac1q,\frac1r\Bigr)$ lies in the (relative) interior of the convex hull of the three points
$\Bigl(\frac1{p_k},\frac1{q_k},\frac1{r_k}\Bigr)$, for $k=0,1,2$.

Then for every $0<\alpha_1,\alpha_2\le\infty$ and $\alpha$ defined by
\[
\frac1\alpha=\frac1{\alpha_1}+\frac1{\alpha_2},
\qquad \frac1\infty:=0,
\]
there exists a constant $C>0$ such that
\[
\|T(f,g)\|_{L^{r,\alpha}}
\le
C\,
C_0^{\theta_0}C_1^{\theta_1}C_2^{\theta_2}\,
\|f\|_{L^{p,\alpha_1}}\,
\|g\|_{L^{q,\alpha_2}},
\]
where $\theta_0,\theta_1,\theta_2\in(0,1)$ are the barycentric coordinates of
$(1/p,1/q,1/r)$ in the simplex spanned by the three endpoints, i.e.
\[
\theta_0+\theta_1+\theta_2=1,\qquad
\frac1p=\sum_{k=0}^2\frac{\theta_k}{p_k},\quad
\frac1q=\sum_{k=0}^2\frac{\theta_k}{q_k},\quad
\frac1r=\sum_{k=0}^2\frac{\theta_k}{r_k}.
\]
\end{theorem}

\begin{lemma}\label{lem:three-point-geometry}
Let $0<s<n$, $1<p,q<\infty$ and define
\[
\frac1r:=\frac1p+\frac1q-\frac{s}{n},
\]
assuming $\frac1p+\frac1q>\frac{s}{n}.$
Then there exists an exponent $p_0\in(1,n/s),$ 
such that the target exponent triple
$\Bigl(\frac1p,\frac1q,\frac1r\Bigr)$
lies in the \emph{relative interior} of the convex hull of three restricted bounded endpoint triples of $\mathcal J_s$.
Consequently, Theorem~\ref{lem:convex-hull-surface} applies to yield for every $0<\alpha_1,\alpha_2\le\infty$,
\[
\mathcal J_s:\ L^{p,\alpha_1}(\mathbb R^n)\times L^{q,\alpha_2}(\mathbb R^n)\to L^{r,\alpha}(\mathbb R^n),
\qquad \frac1\alpha=\frac1{\alpha_1}+\frac1{\alpha_2}.
\]
\end{lemma}
\begin{proof}
Choose $p_0\in(1,n/s)$ such that
\begin{equation}\label{eq:p0-choice-fixed}
\max\Bigl\{\frac1p,\frac1q,\frac{s}{n}\Bigr\}
<
\frac1{p_0}
<
\min\Bigl\{1,\frac1p+\frac1q\Bigr\}.
\end{equation}

Define
\[
P_0=\Bigl(1,1,2-\frac{s}{n}\Bigr),\quad
P_1=\Bigl(\frac1{p_0},0,\frac1{p_0}-\frac{s}{n}\Bigr),\quad
P_2=\Bigl(0,\frac1{p_0},\frac1{p_0}-\frac{s}{n}\Bigr).
\]
By Lemma~\ref{lem:L1LQ}(\ref{item:l1-l1}) together with the pointwise domination
\eqref{eq:riesz-dominates}, the operator $\mathcal J_s$ satisfies weak type bound. Hence $\mathcal J_s$ is of restricted weak type at $P_0$.
By Lemma~\ref{lem:LPLINFTY}, the operator $\mathcal J_s$ satisfies the restricted weak-type bounds at $P_1$ and $P_2$.

Set
\[
\theta_0:=\frac{\left(\frac1p+\frac1q\right)-\frac1{p_0}}{2-\frac1{p_0}},
\qquad
\theta_1:=\frac{\frac1p-\theta_0}{\frac1{p_0}},
\qquad
\theta_2:=\frac{\frac1q-\theta_0}{\frac1{p_0}} .
\]

A direct computation gives
$\theta_0+\theta_1+\theta_2=1$. 

By \eqref{eq:p0-choice-fixed},
$\frac1{p_0}<\frac1p+\frac1q$, hence $\theta_0>0$.
Moreover, $\frac1{p_0}>\frac1q$ implies $\theta_0<\frac1p$, so
$\theta_1>0$. By symmetry, $\theta_2>0$.

The first coordinate equals $\frac1p$, the second equals $\frac1q$, and
$2\theta_0+\frac1{p_0}(\theta_1+\theta_2)
=\frac1p+\frac1q$.

Hence the third coordinate equals
$\frac1p+\frac1q-\frac{s}{n}=\frac1r$.
Therefore,
\[
\theta_0P_0+\theta_1P_1+\theta_2P_2
=\Bigl(\frac1p,\frac1q,\frac1r\Bigr),
\]
and by Theorem \ref{lem:convex-hull-surface}, the conclusion follows.
\end{proof}

\section{Between Two Planes}

In this section, we investigate the boundedness of the bilinear Bessel operator in the case when indices can be placed between the Lebesgue and the Fractional Index Plane, and classify the behaviour at relevant endpoints. 
\subsection{$p=\infty$ or $q=\infty$}
We first recall the bounds for the linear Bessel potential, which can be found in the work of Aronszajn and Smith (see \cite{AroSmi}, Section~10, p.~470).
\begin{theorem}[Bounds for Bessel potentials, Aronszajn and Smith, 1961]
Let $0<s<n$. Then the linear Bessel potential satisfies:
\begin{enumerate}
    \item For every $1\le p\le\infty$,
\[
\|J_s f\|_{L^p}
\le
\|G_s\|_{L^1}\,\|f\|_{L^p};
\]
\item Assume
\[
1\le p<\frac{n}{s},
\qquad
\frac1{r_*}=\frac1p-\frac{s}{n}
\quad
\Bigl(r_*=\frac{np}{n-sp}\Bigr).
\]
Then
\[
\|J_s f\|_{L^{r_*,\infty}}
\le
C_{n,s,p}\,\|f\|_{L^p}.
\]
\end{enumerate}
\end{theorem}

By the pointwise domination, we immediately obtain:
\[
|\mathcal{J}_s(f,g)(x)|
\le \|g\|_{L^\infty}\,(G_s*|f|)(x)
=\|g\|_{L^\infty}\,(J_s|f|)(x). 
\]
As a consequence, we note the following lemma.

\begin{lemma}
\label{lem:LPLINFTY}
Let $0<s<n$, $1\le p<\frac{n}{s}$ and set $r^p_*=\frac{np}{n-sp}$, i.e.
\[
\frac1{r^p_*}=\frac1p-\frac{s}{n}.
\]
Then the following bounds hold:
\[
\mathcal J_s:\; L^p(\mathbb R^n)\times L^\infty(\mathbb R^n)\to L^{r^p_*,\infty}(\mathbb R^n)
\]
 and
\[
\mathcal J_s:\; L^p(\mathbb R^n)\times L^\infty(\mathbb R^n)\to L^p(\mathbb R^n).
\]

Moreover, the same bounds hold after interchanging the two input spaces.
\end{lemma}

\subsection{Triangle}
To complete the picture of boundedness, we next consider the boundedness of the operator into $L^\infty.$
\begin{lemma}[$L^p\times L^q\to L^\infty$ bound below the critical line]
\label{lem:Js_LpLq_Linfty}
Let $0<s<n$ and let $1\le p,q\le \infty$. Assume
\begin{equation}\label{eq:subcritical-Linfty}
0\le \frac1p+\frac1q<\frac{s}{n}.
\end{equation}
Then
\[
\mathcal J_s:
L^p(\mathbb R^n)\times L^q(\mathbb R^n)
\longrightarrow
L^\infty(\mathbb R^n).
\]
\end{lemma}

\begin{proof}
Fix $x\in\mathbb R^n$ and set
$h_x(y):=f(x-y)\,g(x+y)$.

Define $t\in[1,\infty]$ by $\frac1t=\frac1p+\frac1q$,
and let $t'$ be its conjugate exponent, with the convention $1/\infty=0$. 
Then by Hölder's inequality,
\[
\|h_x\|_{L^t}
\le \|f(x-\cdot)\|_{L^p}\,\|g(x+\cdot)\|_{L^q}
=\|f\|_{L^p}\,\|g\|_{L^q},
\]
uniformly in $x$. Hence,
\[
|\mathcal J_s(f,g)(x)|
\le \int_{\mathbb R^n} G_s(y)\,|h_x(y)|\,dy
\le \|G_s\|_{L^{t'}}\,\|h_x\|_{L^t}
\le \|G_s\|_{L^{t'}}\,\|f\|_{L^p}\,\|g\|_{L^q}.
\]
Taking the supremum over $x$ yields the claim. Indeed, by
Lemma~\ref{lem:bessel-kernel-Lr} we have $G_s \in L^{t'}$.
Note that \eqref{eq:subcritical-Linfty} implies the inequality
$t' < \frac{n}{n-s}$.
Consequently, Lemma~\ref{lem:bessel-kernel-Lr} applies and ensures that
$G_s \in L^{t'}$.
\end{proof}

\subsection{Fixed-input interpolation}
We have proved that the bilinear Bessel potential $\mathcal J_s$ satisfies
\[
\|\mathcal J_s(f,g)\|_{L^{r_1}}
\le A\,\|f\|_{L^p}\|g\|_{L^q},\qquad\frac{1}{p}+\frac{1}{q}=\frac{1}{r_1},
\]
and
\[
\|\mathcal J_s(f,g)\|_{L^{r_2,\infty}}
\le B\,\|f\|_{L^p}\|g\|_{L^q},\qquad\frac{1}{p}+\frac{1}{q}=\frac{1}{r_2}+\frac{s}{n}.
\] 
We extend the result to intermediate spaces, by interpolation.
\begin{lemma}
\label{lem:interp-LPLQ}
Let $0<s<n$, $1\le p,q\le\infty$.
Then for every $0<\theta<1$ and every $0<\alpha\le\infty$,
\[
\|\mathcal J_s(f,g)\|_{L^{r_\theta,\alpha}}
\lesssim
A^{1-\theta}B^\theta\,\|f\|_{L^p}\|g\|_{L^q},
\]
where
\[
\frac1{r_\theta}
=
\frac{1-\theta}{r_1}
+
\frac{\theta}{r_2}.
\]
\end{lemma}

\begin{proof}
Let $h=\mathcal J_s(f,g)$.
By hypothesis,
\[
\|h\|_{L^{r_1}}
\le A\|f\|_{L^p}\|g\|_{L^q},
\qquad
\|h\|_{L^{r_2,\infty}}
\le B\|f\|_{L^p}\|g\|_{L^q}.
\]

Let $h^*$ denote the decreasing rearrangement of $h$.

From Chebyshev's inequality and the strong $L^{r_1}$ bound,
\[
h^*(t)
\le
\|h\|_{L^{r_1}}\,t^{-1/r_1}
\le
A\|f\|_{L^p}\|g\|_{L^q}\,t^{-1/r_1},
\qquad t>0.
\]
Similarly, from the weak-type bound,
\[
h^*(t)
\le
\|h\|_{L^{r_2,\infty}}\,t^{-1/r_2}
\le
B\|f\|_{L^p}\|g\|_{L^q}\,t^{-1/r_2},
\qquad t>0.
\]
Therefore,
\[
h^*(t)
\le
\min\Bigl(
A\|f\|_{L^p}\|g\|_{L^q}\,t^{-1/r_1},
\;
B\|f\|_{L^p}\|g\|_{L^q}\,t^{-1/r_2}
\Bigr).
\]

Let $0<\theta<1$ and define $r_\theta$ by
$\frac1{r_\theta}
=
\frac{1-\theta}{r_1}
+
\frac{\theta}{r_2}$.

Let $t_0>0$ be determined by
\[
A t_0^{-1/r_1}=B t_0^{-1/r_2},\qquad
t_0^{\,1/r_2-1/r_1}=\frac{A}{B}.
\]

Assume first that $0<\alpha<\infty$.
\begin{align*}
\|h\|_{L^{r_\theta,\alpha}}^\alpha&=
\int_0^\infty
\bigl(t^{1/r_\theta}h^*(t)\bigr)^\alpha
\frac{dt}{t}\\
&\le
\int_0^{t_0}
\bigl(B\|f\|_{L^p}\|g\|_{L^q}\,
t^{1/r_\theta-1/r_2}\bigr)^\alpha
\frac{dt}{t}\\
&+
\int_{t_0}^{\infty}
\bigl(A\|f\|_{L^p}\|g\|_{L^q}\,
t^{1/r_\theta-1/r_1}\bigr)^\alpha
\frac{dt}{t}.
\end{align*}

Since $r_1<r_2$
\[
\frac1{r_\theta}-\frac1{r_2}>0,
\qquad
\frac1{r_\theta}-\frac1{r_1}<0,
\]
both integrals converge.

A direct computation gives
\[
\|h\|_{L^{r_\theta,\alpha}}^\alpha
\le C
(A\|f\|_{L^p}\|g\|_{L^q})^{\alpha(1-\theta)}
(B\|f\|_{L^p}\|g\|_{L^q})^{\alpha\theta},
\]
where $C=C(n,s,p,q)$ a constant determined only on $n,s,p,q$.
Taking $\alpha$-th roots yields
\[
\|h\|_{L^{r_\theta,\alpha}}
\lesssim
A^{1-\theta}B^\theta
\|f\|_{L^p}\|g\|_{L^q}.
\]

When $\alpha=\infty$, we use
\[
\|h\|_{L^{r_\theta,\infty}}
=
\sup_{t>0} t^{1/r_\theta}h^*(t),
\]
together with the same pointwise bound.
The supremum is attained up to constants at $t\approx t_0$,
which gives the same estimate. \end{proof}

\section{Sharpness and Counterexamples}

In this section, we construct concrete counterexamples to show the failure of boundedness of operators on a line we call a critical line. We then establish sharpness for previous fractional index bounds.
\subsection{Failure on the critical line}
We first prove failure of boundedness on the critical line, where for $0<s<n,$ and $1\leq p,q\leq \infty,$ $\frac{1}{p}+\frac{1}{q}=\frac{s}{n}.$
\begin{proposition}[Failure of $L^p\times L^q\to L^\infty$ on the critical line]
\label{prop:fail-critical-Linfty}
Let $0<s<n$ and let $1\le p,q\le\infty$ satisfy
\begin{equation}\label{eq:critical-line}
\frac1p+\frac1q=\frac{s}{n},
\end{equation}
excluding the trivial case $(p,q)=(\infty,\infty)$.
Then the bilinear Bessel potential
\[
\mathcal{J}_s(f,g)(x)
=\int_{\mathbb R^n} G_s(y)\,f(x-y)\,g(x+y)\,dy
\]
does \emph{not} define a bounded operator
\[
\mathcal{J}_s:\ L^p(\mathbb R^n)\times L^q(\mathbb R^n)\longrightarrow L^\infty(\mathbb R^n).
\]
\end{proposition}

\begin{proof}
We use the fact that $G_s$ is nonnegative and satisfies the local lower bound
\begin{equation}\label{eq:Gs-lower-critical}
G_s(y)\ge c\,|y|^{s-n}\qquad(0<|y|<1)
\end{equation}
for some $c>0$.

\medskip
\noindent\textbf{Case 1: $1<p,q<\infty$.}
Choose $\beta,\gamma>0$ such that
\[
\beta p>1,\qquad \gamma q>1,\qquad \beta+\gamma\le 1.
\]
(This is possible since $1/p+1/q=s/n<1$, hence $1/p+1/q<1$ and we may take
$\beta\downarrow 1/p$ and $\gamma\downarrow 1/q$ so that $\beta+\gamma\le 1$.)

Define
\[
f(y):=|y|^{-n/p}\Bigl(\log\tfrac{e}{|y|}\Bigr)^{-\beta}\mathbf 1_{|y|<1}(y),
\qquad
g(y):=|y|^{-n/q}\Bigl(\log\tfrac{e}{|y|}\Bigr)^{-\gamma}\mathbf 1_{|y|<1}(y).
\]
Then $f\in L^p$ and $g\in L^q$. Evaluating at $x=0$ and using
\eqref{eq:Gs-lower-critical},
\begin{align*}
\mathcal{J}_s(f,g)(0)
&=\int_{\mathbb R^n} G_s(y)\,f(-y)\,g(y)\,dy \\
&\ge c\int_{|y|<1}
|y|^{s-n}\,|y|^{-n/p}\,|y|^{-n/q}
\Bigl(\log\tfrac{e}{|y|}\Bigr)^{-(\beta+\gamma)}\,dy.
\end{align*}
Using \eqref{eq:critical-line}, we have
$s-n-n/p-n/q=-n$, and so
\[
\mathcal{J}_s(f,g)(0)\ge c\int_{|y|<1}
|y|^{-n}\Bigl(\log\tfrac{e}{|y|}\Bigr)^{-(\beta+\gamma)}\,dy.
\]
In polar coordinates the expression is comparable to
\[
\int_0^1 \frac{dr}{r\bigl(\log(e/r)\bigr)^{\beta+\gamma}},
\]
which diverges since $\beta+\gamma\le 1$. Hence $\mathcal{J}_s(f,g)(0)=\infty$,
and we can conclude that $\mathcal{J}_s(f,g)\notin L^\infty$.

\medskip
\noindent\textbf{Case 2: $(p,q)=(n/s,\infty)$.}
Suppose, for the sake of a contradiction that
\[
\|\mathcal{J}_s(f,g)\|_{L^\infty}\le C\|f\|_{L^{n/s}}\|g\|_{L^\infty}.
\]
Taking $g\equiv 1$ gives
\begin{equation}\label{eq:crit-linear-assumed}
\|G_s*f\|_{L^\infty}\le C\|f\|_{L^{n/s}}.
\end{equation}
Fix $0<\rho<1/2$ and set
\[
f(y):=|y|^{-s}\Bigl(\log\tfrac{e}{|y|}\Bigr)^{-1}\mathbf 1_{|y|<\rho}(y).
\]
A direct polar-coordinate computation shows $f\in L^{n/s}$. On the other hand,
by \eqref{eq:Gs-lower-critical},
\begin{align*}
(G_s*f)(0)
&\ge c\int_{|y|<\rho} |y|^{s-n}\,|y|^{-s}\Bigl(\log\tfrac{e}{|y|}\Bigr)^{-1}\,dy \\
&= c\int_{|y|<\rho} |y|^{-n}\Bigl(\log\tfrac{e}{|y|}\Bigr)^{-1}\,dy
\sim \int_0^\rho \frac{dr}{r\log(e/r)}=\infty,
\end{align*}
contradicting \eqref{eq:crit-linear-assumed}. Thus the endpoint boundedness fails.

\medskip
\noindent\textbf{Case 3: $(p,q)=(\infty,n/s)$.}
The contradiction follows by symmetry, i.e. by swapping the roles of $f$ and $g$.\end{proof}

\subsection{Sharp Lorentz indices}
Next, we prove fractional index sharpness, in the cases given by the following lemma.

\begin{lemma}[Sharpness of fractional index bounds] 
Let $0<s<n.$ For $1\leq p,q\leq \infty,$ let 
\begin{equation}\frac{1}{r}=\frac{1}{p}+\frac{1}{q}-\frac{s}{n}.\end{equation}
The following hold: 
\begin{enumerate}
    \item \label{item:bessel-inside} Let $1<p,q<\infty, $ 
    Then for every $0<\alpha,$ $\frac{1}{\alpha}>\frac{1}{p}+\frac{1}{q},$ 
    \[
    \mathcal{J}_s:L^p(\mathbb{R}^n)\times L^q(\mathbb{R}^n)\not\to L^{r,\alpha}(\mathbb{R}^n).
    \]
    The bound $\frac{1}{\alpha}=\frac{1}{p}+\frac{1}{q}$ is sharp in Lemma~\ref{lem:three-point-geometry}.
    \item \label{item:bessel-endpoints} Let $1<p<\frac{n}{s}$, and fix $q\in\{1,\infty\}$. Define $r$ by
    \begin{equation*}\label{eq:r-def-bessel}
    \frac{1}{r}=\frac1p+\frac1q-\frac{s}{n}.
    \end{equation*}
    Then for every $0<\alpha<p,$ 
    \[
    \mathcal J_s:L^p(\mathbb{R}^n)\times L^q(\mathbb{R}^n)\not\to L^{r,\alpha}(\mathbb R^n).
    \]
    In particular, the Lorentz index restrictions $\alpha\ge p$ in the cases
    $q=\infty$ (Lemma~\ref{lem:LPLINFTY}) and $q=1$ (Lemma~\ref{lem:L1LQ}) are sharp.
    \item \label{item:bessel-enpoint-bi} Let $p=1,$ $q=\infty.$ Then for any $0<\alpha<\infty,$
    \[
    \mathcal J_s:L^1(\mathbb{R}^n)\times L^\infty(\mathbb{R}^n)\not\to L^{\frac{n}{n-s},\alpha}(\mathbb R^n)
    \]
\end{enumerate}

\end{lemma}
\begin{proof} 
For $0<p,q,r<\infty,$ $\beta,\delta>0,$ we define
   \begin{gather}
    f(x)\coloneq |x|^{-\frac{n}{p}}\log\left(\frac{e}{|x|}\right)^{-\frac{1}{\beta}}\mathbf{1}_{B(0,1)\setminus \{0\}}(x); \label{eq:lp-function}\\
    h(x)\coloneq |x|^{-\frac{n}{r}}\log\left(\frac{e}{|x|}\right)^{-\frac{1}{\delta}}\mathbf{1}_{B(0,\frac{1}{8})\setminus\{0\}}(x).\label{eq:lr-function}
    \end{gather}

For $x\in \mathbb{R}^n$ we use $B(x,r)$ to refer to the ball of radius $r,$ centered at $x,$ and for a set $A,$ we use $\mathbf{1}_{A}$ to be the indicator function of the set $A,$ i.e. the function such that $\mathbf{1}_{A}(x)=1$ if $x\in A,$ and $\mathbf{1}_{A}(x)=0$ otherwise. 

\noindent\textbf{\textit{(1)} $1<p,q<\infty$:}

    Since $\frac{1}{\alpha}>\frac{1}{p}+\frac{1}{q},$ and $1<p,q< \infty,$ we can find $0<\beta, \gamma< \infty$ such that $\beta <p,$ $\gamma <q,$ and $\frac{1}{\alpha}=\frac{1}{\beta}+\frac{1}{\gamma}.$  Let $f$ be as in \ref{eq:lp-function}, and let $h$ be as in \eqref{eq:lr-function}, with $\delta = \alpha.$
    Let \[
     g(x)\coloneq |x|^{-\frac{n}{q}}\log\left(\frac{e}{|x|}\right)^{-\frac{1}{\gamma}} \mathbf{1}_{B(0,1)\setminus\{0\}}(x); 
    \]

Since $\beta<p$ and $\gamma<q$, it follows that $f\in L^p = L^{p,p}$ and $g\in L^q = L^{q,q}$. However, because
\[
\frac{1}{\beta}+\frac{1}{\gamma}=\frac{1}{\alpha},
\]
the function $h$ does not belong to $L^{r,\alpha}$ (see Exercise~1.4.8 in \cite{Gra}).

Let $|x|<\frac{1}{8},$ $x\neq 0.$ The bilinear Bessel potential is given by 
\[
\mathcal{J}_s(f,g)(x)=\int_{\mathbb{R}^n}G_s(y)f(x-y)g(x+y)dy.
\]

Using a change of variables, we have 
\[
\mathcal{J}_s(f,g)(x)=\int_{\mathbb{R}^n}G_s(x-y)f(y)g(2x-y)dy.
\]

Since the integrand is nonnegative, we restrict to the region 
\begin{equation} \label{eq:set-e-x}
E_x\coloneq \{y\in \mathbb{R}^n: |x|<|y|<\frac{3}{2}|x|\},
\end{equation}
where $|x|<\frac{1}{8}.$

Note that for $y\in E_x$ and $|x|<\frac{1}{8},$
\begin{gather*}
|x-y|, |x+y|\leq |x|+|y|< 1;\\
|2x-y|\leq 2|x|+|y|< 1.
\end{gather*}

Therefore, 

\[
\mathcal{J}_s(f,g)(x)\gtrsim \int_{E_x}|x-y|^{s-n}|y|^{-\frac{n}{p}}\log^{-\frac{1}{\beta}}\left(\frac{e}{|y|}\right)|2x-y|^{-\frac{n}{q}}\log^{-\frac{1}{\gamma}}\left(\frac{e}{|2x-y|}\right)dy.
\]

We observe that if $1>|z_1|\geq|z_2|,$ and $a>0,$ it follows that
\begin{equation} \label{eq:log-preserves-order}
    \log^{-\frac{1}{a}}\left(\frac{e}{|z_1|}\right)\geq \log^{-\frac{1}{a}}\left(\frac{e}{|z_2|}\right). 
\end{equation}

For $x$ such that $|x|<\frac{1}{8},$ and $y\in E_x,$ $|y|\geq |x|$ and $|2x-y|\geq 2|x|-\frac{3}{2}|x|\geq \frac{1}{2}|x|.$ Therefore, by the equation \eqref{eq:log-preserves-order}, the following hold:
\begin{gather} \label{eq:log-terms1}
    \log^{-\frac{1}{\beta}}\left(\frac{e}{|y|}\right)\geq \log^{-\frac{1}{\beta}}\left(\frac{e}{|x|}\right);\\ \label{eq:log-terms2}
    \log^{-\frac{1}{\gamma}}\left(\frac{e}{|2x-y|}\right)\geq \log^{-\frac{1}{\gamma}}\left(\frac{2e}{|x|}\right)\gtrsim \log^{-\frac{1}{\gamma}}\left(\frac{e}{|x|}\right).
\end{gather}

For $y\in E_x,$ we have
\begin{gather} \label{eq:absolute-terms1}
    |2x-y|^{-\frac{n}{q}}\gtrsim|y|^{-\frac{n}{q}};\\
    |x-y|^{s-n}\gtrsim |y|^{s-n}, \label{eq:absolute-terms2}
\end{gather}
where the implied constants depend on $n,s$ and $q.$

Combining the equations \eqref{eq:log-terms1}, \eqref{eq:log-terms2}, \eqref{eq:absolute-terms1} and \eqref{eq:absolute-terms2}, it follows that 
\begin{align*}
\mathcal{J}_s(f,g)(x)&\gtrsim \log^{-(\frac{1}{\beta}+\frac{1}{\gamma})}\left(\frac{e}{|x|}\right) \int_{E_x} |y|^{s-n(1+\frac{1}{p}+\frac{1}{q})}dy \\
&\approx \log^{-\frac{1}{\alpha}}\left(\frac{e}{|x|}\right) \int_{|x|}^{\frac32|x|} \rho^{-\frac{n}{r}-1} d\rho \\
&= \log^{-\frac{1}{\alpha}}\left(\frac{e}{|x|}\right)|x|^{-\frac{n}{r}} \frac{r}{n}\left(1-\left(\frac{3}{2}\right)^{-\frac{n}{r}}\right) \\
&\gtrsim \log^{-\frac{1}{\alpha}}\left(\frac{e}{|x|}\right)|x|^{-\frac{n}{r}},
\end{align*}
where the fact that $1>\frac{2}{3}^{\frac{n}{r}},$ implies the last inequality. 

Hence,  for $x\in \mathbb{R}^n,$ such that $x\neq 0$ and $|x|<\frac{1}{8},$ $\mathcal{J}_s(f,g)(x)\geq h(x),$ and since $h\notin L^{r,\alpha},$ it follows that $\mathcal{J}_s(f,g)\notin L^{r,\alpha}$ completing the proof in the case \ref{item:bessel-inside}.

\noindent\textbf{\textit{(2)}: $1<p<\frac{n}{s}$ and $q\in\{1,\infty\}:$}  Fix $0<\alpha<p$ and let $f$ be as in equations \eqref{eq:lp-function} and let $h$ be as in \eqref{eq:lr-function}, with $\delta=\alpha$, and let $g(x)\coloneq\mathbf 1_{B(0,4)}(x).$ Note that $g\in L^q(\mathbb R^n)$ for $q\in\{1,\infty\}$, and as before, since $\alpha<p,$  $f\in L^p= L^{p,p}.$ The function $h$ does not belong to $L^{r,\alpha}(\mathbb R^n)$ whenever
$0<\alpha<p$ (see Exercise~1.4.8 in \cite{Gra}).

The proof now follows similarly to the proof of item i). Namely, we observe that for $|x|<\frac{1}{8},$ and $y\in E_x$ defined as in \eqref{eq:set-e-x}, $|2x-y|<1<4,$
and the following inequality holds:
\begin{align*}
\mathcal{J}_s(f,g)(x)&=\int_{\mathbb{R}^n}G_s(x-y)f(y)g(2x-y)dy \\
&\gtrsim  \int_{E_x}|x-y|^{s-n}|y|^{-\frac{n}{p}}\log^{-\frac{1}{\alpha}}\left(\frac{e}{|y|}\right)dy.\end{align*}
Following the same reasoning as in item \ref{item:bessel-inside}, it follows that 
\[
\mathcal{J}_s(f,g)(x) \gtrsim h(x) \quad \text{for } x\in B(0,1/8)\setminus\{0\}.
\]
Therefore, $\mathcal{J}_{s}(f,g)\notin L^{r,\alpha}(\mathbb{R}^n).$

\noindent\textbf{\textit{(3)}: $p=1,\ q=\infty$.}
We argue by contradiction. Assume that for some $0<\alpha<\infty$,
\[
\mathcal J_s:
L^1(\mathbb R^n)\times L^\infty(\mathbb R^n)
\longrightarrow
L^{\frac{n}{n-s},\alpha}(\mathbb R^n)
\]
is bounded. Let $g\equiv 1$ and, for $\varepsilon>0$, define
\[
f_\varepsilon \coloneq \frac{\mathbf 1_{B(0,\varepsilon)}}{|B(0,\varepsilon)|}.
\]
Then $\|f_\varepsilon\|_{L^1}=1$ and $\|g\|_{L^\infty}=1$, so by assumption,
\begin{equation}\label{eq:bessel-bound-3}
\|\mathcal J_s(f_\varepsilon,1)\|_{L^{\frac{n}{n-s},\alpha}}
\lesssim 1,
\qquad\text{uniformly in }\varepsilon.
\end{equation}

For all $x\in\mathbb R^n$,
\[
\mathcal J_s(f_\varepsilon,1)(x)
=\int_{\mathbb R^n} G_s(y) f_\varepsilon(x-y)\,dy
=(G_s*f_\varepsilon)(x).
\]
Fix $0<\varepsilon<\tfrac18$ and consider $x$ such that
$4\varepsilon\le |x|\le \tfrac12$.
For any $y\in B(0,\varepsilon)$ we have
\[
|x-y|\le |x|+\varepsilon \le \tfrac54 |x|,
\qquad
|x-y|\ge |x|-\varepsilon \ge \tfrac34 |x|.
\]
Since $|x-y|\le 1$ in this region and the Bessel kernel satisfies
$G_s(z)\gtrsim |z|^{s-n}$ for $|z|\le 1$, it follows that
\[
(G_s*f_\varepsilon)(x)
=\frac1{|B(0,\varepsilon)|}\int_{B(0,\varepsilon)} G_s(x-y)\,dy
\gtrsim
\frac1{|B(0,\varepsilon)|}\int_{B(0,\varepsilon)} |x-y|^{s-n}\,dy
\gtrsim |x|^{s-n}.
\]
Hence,
\[
\mathcal J_s(f_\varepsilon,1)(x)
\gtrsim |x|^{s-n}\,
\mathbf 1_{\{\,4\varepsilon\le |x|\le \frac12\,\}}(x).
\]

By monotonicity of Lorentz quasi-norms with respect to pointwise domination,
\[
\|\mathcal J_s(f_\varepsilon,1)\|_{L^{\frac{n}{n-s},\alpha}}
\gtrsim
\bigl\||x|^{s-n}\mathbf 1_{\{\,4\varepsilon\le |x|\le \frac12\,\}}\bigr\|_
{L^{\frac{n}{n-s},\alpha}}.
\]

Let $r=\frac{n}{n-s}$. For $\lambda>0$, define
\[
E_\lambda
=
\Bigl\{x:\ 4\varepsilon\le |x|\le \tfrac12,\ |x|^{s-n}>\lambda\Bigr\}.
\]
For $\lambda\in[2^{\,n-s},(8\varepsilon)^{-(n-s)}]$
(equivalently $\lambda^{-1/(n-s)}\in[8\varepsilon,\tfrac12]$),
we have
\[
E_\lambda
=
\Bigl\{x:\ 4\varepsilon\le |x|<\lambda^{-1/(n-s)}\Bigr\},
\]
and hence
\[
|E_\lambda|
=
c_n\bigl(\lambda^{-r}-(4\varepsilon)^n\bigr)
\gtrsim \lambda^{-r},
\]
with constants independent of $\varepsilon$.
Indeed, for $\lambda \le (8\varepsilon)^{-(n-s)}$ one has
$\lambda^{-1/(n-s)} \ge 8\varepsilon$, hence $\lambda^{-r} \ge (8\varepsilon)^n$.
Therefore
\[
\lambda^{-r}-(4\varepsilon)^n
\ge (8\varepsilon)^n-(4\varepsilon)^n
=(1-2^{-n})(8\varepsilon)^n
\ge (1-2^{-n})\,\lambda^{-r},
\]
which shows that $|E_\lambda|\gtrsim \lambda^{-r}$ with constants independent of
$\varepsilon$.
Using the definition of the Lorentz quasi-norm,
\[
\bigl\||x|^{s-n}\mathbf 1_{\{\,4\varepsilon\le |x|\le \frac12\,\}}\bigr\|_
{L^{r,\alpha}}^{\alpha}
\gtrsim
\int_{2^{\,n-s}}^{(8\varepsilon)^{-(n-s)}}
\Bigl(\lambda\,|E_\lambda|^{1/r}\Bigr)^{\alpha}\,\frac{d\lambda}{\lambda}
\gtrsim
\int_{2^{\,n-s}}^{(8\varepsilon)^{-(n-s)}} \frac{d\lambda}{\lambda}
\sim \log\frac1\varepsilon.
\]
Letting $\varepsilon\to0$ yields
\[
\bigl\||x|^{s-n}\mathbf 1_{\{\,4\varepsilon\le |x|\le \frac12\,\}}\bigr\|_
{L^{r,\alpha}}
\longrightarrow \infty,
\]
which contradicts the assumed uniform bound in equation \eqref{eq:bessel-bound-3}. This completes
the proof.
 \end{proof}

\section{Geometry of the Exponent Region}

In this section we provide a schematic picture in the $(1/p,1/q,1/r)$-space. Since $p$ and $q$
are fixed, the interpolation segment is vertical (only $1/r$ changes).
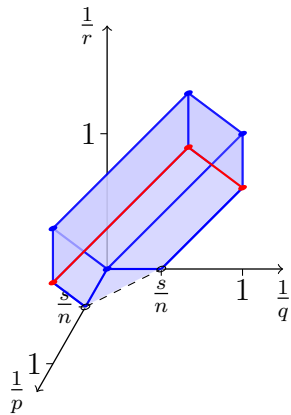
\begin{center}
\begin{tikzpicture}[
    scale=1.8,
    x={(1cm,0cm)},      
    y={(0.4cm,0.7cm)},  
    z={(0cm,1cm)}       
]

\def\sn{0.4} 

\coordinate (O) at (0,0,0);
\coordinate (X) at (1.3,0,0);
\coordinate (Y) at (0,-1.3,0);
\coordinate (Z) at (0,0,1.8);

\draw[->,line width=0.5pt] (O)--(X) node[below] {$\frac{1}{q}$};
\draw[->,line width=0.5pt] (O)--(Y) node[left]  {$\frac{1}{p}$};
\draw[->,line width=0.5pt] (O)--(Z) node[left]  {$\frac{1}{r}$};

\draw (1,0,0) -- (1,0,-0.05);
\node[below] at (1,0,0) {$1$};

\draw (0,-1,0) -- (-0.05,-1,0);
\node[left] at (0,-1,0) {$1$};

\draw (0,0,1) -- (-0.05,0,1);
\node[left] at (0,0,1) {$1$};

\draw (\sn,0,0) -- (\sn,0,-0.05);
\node[below] at (\sn,0,0) {$\tfrac{s}{n}$};

\draw (0,-\sn,0) -- (-0.05,-\sn,0);
\node[left] at (0,-\sn,0) {$\tfrac{s}{n}$};


\coordinate (P0)  at (0,0,0);  
\coordinate (P01) at (0,-1,1);  
\coordinate (P10) at (1,0,1);  
\coordinate (P11) at (1,-1,2);  

\coordinate (Q1) at (0,-\sn,0);        
\coordinate (Q2) at (\sn,0,0);        
\coordinate (Q3) at (0,-1,1-\sn);      
\coordinate (Q4) at (1,0,1-\sn);      
\coordinate (Q5) at (1,-1,2-\sn);      

\filldraw[
  fill=blue!20,
  fill opacity=0.15,
  draw=red!0
]
(Q1) -- (Q2) -- (Q4) -- (Q5) -- (Q3) -- cycle;

\filldraw[
  fill=blue!20,
  draw=blue,
  opacity=0.7,
  line width=0.8pt
]
(P0) -- (P01) -- (P11) -- (P10) -- cycle;

\filldraw[
  fill=blue!20,
  draw=blue,
  opacity=0.7,
  line width=0.8pt
]
(Q1) -- (P0) -- (P01) -- (Q3) -- cycle;

\filldraw[
  fill=blue!20,
  draw=blue,
  opacity=0.7,
  line width=0.8pt
]
(Q5) -- (P11) -- (P01) -- (Q3) -- cycle;

\filldraw[
  fill=blue!20,
  draw=blue,
  opacity=0.7,
  line width=0.8pt
]
(Q5) -- (P11) -- (P10) -- (Q4) -- cycle;

\filldraw[
  fill=blue!20,
  draw=blue,
  opacity=0.7,
  line width=0.8pt
]
(Q2) -- (Q4) -- (P10) -- (P0) -- cycle;

\filldraw[
  fill=blue!20,
  draw=blue!0,
  opacity=0.7,
  line width=0.8pt
]
(P0) -- (Q1) -- (Q2) -- cycle;

\draw[blue, line width=0.8pt] (Q1) -- (Q3);
\draw[blue, line width=0.8pt] (Q4) -- (Q2);
\draw[red, line width=0.8pt] (Q4) -- (Q5);
\draw[red, line width=0.8pt] (Q3) -- (Q5);
\draw[black, dashed] (Q1) -- (Q2);
\draw[blue, line width=0.8pt] (Q1) -- (P0);
\draw[blue, line width=0.8pt] (Q2) -- (P0);
\foreach \P in {P0,P01,P10,P11}
    \fill[blue] (\P) circle (0.03);
\foreach \Q in {Q1,Q2}
    \draw[black] (\Q) circle (0.03);
\foreach \Q in {Q3,Q4,Q5}
    \fill[red] (\Q) circle (0.03);

\end{tikzpicture}
\end{center}

This picture depicts the region of the exponents $(1/p,1/q,1/r)$ for which boundedness holds. The \emph{blue solid region} indicates
the part of the exponent region where strong-type boundedness is established in
this paper, while the \emph{red region} indicates boundary portions where weak-type
boundedness is established. Dashed lines and hollow circles indicate endpoint or
boundary behavior where boundedness fails or is excluded by the counterexamples proved above.

\textbf{Acknowledgements.} The authors would like to thank Professor Loukas Grafakos for useful conversations throughout the process of writing the paper.

\Addresses
\end{document}